\documentclass{amsart}
\usepackage{color}

\usepackage{amsmath,amssymb,amsfonts,url,epsf,epsfig}
\usepackage{euscript,graphicx}
\usepackage{enumitem}

\usepackage{cancel}

\newtheorem{theorem}{Theorem}[section]
\newtheorem{lemma}[theorem]{Lemma}

\newtheorem{corollary}[theorem]{Corollary}

\newtheorem{conjecture}[theorem]{Conjecture}
\newtheorem{definition}[theorem]{Definition}

\newtheorem{question}[theorem]{Question}
\newtheorem{remark}[theorem]{Remark}

\newcommand{\Aut}{{\rm Aut}}

\newcommand{\Sym}{\hbox{{\rm Sym}}}

\newcommand{\Cos}{\hbox{{\rm Cos}}}

\renewcommand{\phi}{\varphi}

\newcommand{\V}{{\rm V}}
\newcommand{\A}{{{\rm A}}}
\newcommand{\E}{{{\rm E}}}
\newcommand{\T}{{{\rm T}}}
\newcommand{\N}{{{\rm N}}}

\newcommand{\R}{{\mathcal{R}}}
\newcommand{\G}{{\mathcal{L}}}
\newcommand{\X}{{\mathcal{X}}}
\newcommand{\cT}{{\mathcal{T}}}

\newcommand{\la}{\langle}
\newcommand{\ra}{\rangle}

\newcommand{\AGL}{{\hbox{\rm AGL}}}

\newcommand{\C}{{\mathcal{C}}}
\newcommand{\rC}{{\rm{C}}}
\newcommand{\tX}{{\tilde{X}}}
\newcommand{\tG}{{\tilde{G}}}
\newcommand{\tA}{{\tilde{A}}}
\newcommand{\tGa}{{\tilde{\Gamma}}}

\newcommand{\tv}{{\tilde{v}}}

\begin{document}

\title[Linking ring structures]{Resolution of a conjecture about linking ring structures}

\author[Marston Conder]{Marston Conder}
\address{Marston Conder,\newline
Department of Mathematics, University of Auckland, Private bag 92019, Auckland 1142, New Zealand}
 \email{m.conder@auckland.ac.nz}

\author[Luke Morgan]{Luke Morgan}
\address{Luke Morgan,\newline
University of Primorska, UP FAMNIT, Glagolja\v{s}ka 8, 6000 Koper, Slovenia, and University of Primorska, UP IAM, Muzejski trg 2,  6000 Koper, Slovenia. \newline
\emph{Current affiliation}: Centre for the Mathematics of Symmetry and Computation, Department of Mathematics and Statistics, University of Western Australia, 35 Stirling Highway, Crawley, Western Australia 6009}
 \email{luke.morgan@uwa.edu.au}

\author[P.\ Poto\v{c}nik]{Primo\v{z} Poto\v{c}nik}
\address{Primo\v{z} Poto\v{c}nik,\newline
Faculty of Mathematics and Physics, University of Ljubljana,  Jadranska 19, SI-1000 Ljubljana, Slovenia;
\newline affiliated with: 
Institute of Mathematics, Physics and Mechanics,
 Jadranska 19, SI-1000 Ljubljana, Slovenia.
 } 
\email{primoz.potocnik@fmf.uni-lj.si}

\thanks{We are grateful to the referee for their suggestions and corrections. The first author is grateful to New Zealand's Marsden Fund for its support via project UOA2030. 
The  second author acknowledges the Australian Research Council grant DE160100081 and the Slovenian Research Agency research programme P1-0285 and research projects J1-1691, N1-0160, J1-2451, N1-0208. 
The third author gratefully acknowledges the support of Slovenian Research Agency, programme P1--0294 and research project N1-0126.}


\begin{abstract}
An LR-structure is a tetravalent vertex-transitive graph together with a special type of a decomposition of its
edge-set into cycles. LR-structures were introduced in a paper by P.~Poto\v{c}nik and S.~Wilson, titled `Linking rings structures and tetravalent semisymmetric graphs', in {\em Ars Math.\ Contemp.}~{\bf 7} (2014), as a tool to study tetravalent semisymmetric graphs of girth $4$.
In this paper, we use the methods of group amalgams to resolve some problems left open in the above-mentioned paper.
\end{abstract}

\maketitle


\section{Introduction}
\label{sec:intro}

This paper is about a combinatorially interesting class of
connected tetravalent vertex-transitive graphs, which can be characterised
by the existence of a particularly nice decomposition of the edges of the graph into cycles.
This class of graphs and the corresponding edge-decompositions, called
{\em LR-structures}, were introduced in \cite{LR1}
  and later studied in \cite{LR2,LR3}. In \cite{LR1} two 
intriguing problems were left open (see Question~\ref{q} and Conjecture~\ref{c})
which, even though entirely combinatorial in their nature, defied attempts 
to solve them with purely graph-theoretical tools.

The main aim of this paper is to provide a solution of each of these two problems. The approach we take 
is based on a study of $2$-arc-transitive groups of automorphisms of tetravalent graphs, the algebraic structure of which 
can be deduced from the works of Gardiner \cite{Gar}, Weiss \cite{weiss} and the third author \cite{Pot},  
and then exploited with the help of some computations. 

We begin with motivation that leads naturally to the notions of LR-groups and LR-structures.

First, consider the pairs $(\Gamma,G)$ where $\Gamma$ is a connected (but not necessarily finite) tetravalent 
graph and $G$ is a vertex-transitive group of automorphisms of $\Gamma$. A rough classification of these 
can be obtained by considering the permutation group $G_v^{\Gamma(v)}\!$ induced by the action of the vertex-stabiliser $G_v$ on
the neighbourhood $\Gamma(v)$ of a vertex $v$. 
 
 If $G_v^{\Gamma(v)}$ is transitive, then $G$ acts transitively on the set $\A(\Gamma)$ of arcs (ordered pairs of  adjacent vertices) of $\Gamma$.
Arc-transitive graphs have been studied extensively for decades, and many aspects of the tetravalent case are now well understood. For example, the structure of their automorphism groups is known \cite{D4,Pot}, a complete list of examples of order up to 640 has been constructed \cite{cubiccensus}, and several infinite families have been found and analysed \cite{GP}.
 
 The next natural case to consider is when the group $G_v^{\Gamma(v)}$ has two orbits.
 This case splits into a number of subcases.
 If one of the orbits of $G_v^{\Gamma(v)}$ has length $1$ (and the other has length $3$), then one can study the graph $\Gamma$
 through the associated $6$-valent $G$-arc-transitive graph $\Gamma'$ obtained by merging
 pairs $\{u,w\}$ of vertices of $\Gamma$ with the property that $G_v=G_w$. 
 A related approach 
  was used in \cite{cubiccensus} to study vertex- but not arc-transitive $3$-valent graphs
 using some theory of $4$-valent arc-transitive graphs.
 
A much more complex situation arises combinatorially if $G_v^{\Gamma(v)}$ has two orbits of length $2$. 
In this case $G_v^{\Gamma(v)}$ is permutation isomorphic to either the cyclic group $\rC_2 = \langle (1,2)(3,4)\rangle$ or the
Klein group $\V_4 = \langle (1,2), (3,4) \rangle$, in their
 respective intransitive actions on $4$ points. 
In these two cases the group $G$ has precisely two orbits on the arc-set $\A(\Gamma)$, but can have either one or two orbits on the edge-set 
$\E(\Gamma)$. 

If  $G$ has a single orbit on $\E(\Gamma)$, then the graph $\Gamma$ belongs to the widely-studied class of graphs admitting a half-arc-transitive group action,
 which has received much attention in the recent past; see \cite{PraSpiGlobal,ConPotSpa,MikSpaWil,HAT4,RivSpa,SpiHAT,SpiXiaHAT,Xia} for example.
On the other hand, the case where $G$ has two orbits on  $\E(\Gamma)$ has
received much less attention so far. Here the analysis can again be split into two subcases, depending on
whether $G_v^{\Gamma(v)}$ is isomorphic to $\rC_2$ or to $\V_4$. In the first subcase one can easily see that the vertex-stabiliser $G_v$ 
is itself
 isomorphic to $\rC_2$, 
which forces the group $G$ to be relatively small (indeed $|G| = 2|\V(\Gamma)|$, to be precise), which allows the use of a number of standard group-theoretical approaches. 

In this paper, we are interested in the remaining subcase, where $G_v^{\Gamma(v)} \cong \V_4$, and $G$ is intransitive on $\E(\Gamma)$. 
Accordingly, we are interested in the situation covered by the following definition:
 \begin{definition}
 \label{def:LRgroup}
Let $G$ be a vertex-transitive group of automorphisms of a connected tetravalent graph $\Gamma$ such that $G_v^{\Gamma(v)}$ is permutation isomorphic to the Klein group
$\V_4$ in its faithful intransitive action on four points. If $G$ has {\em two} orbits on $\E(\Gamma)$,
then we call the group $G$ an {\em LR-group} of automorphisms of $\Gamma$.

 For a subgroup $X$ of $ \Aut(\Gamma)$ that contains $G$, we say that $G$ is a \emph{maximal LR-subgroup} of  $X $ if   there is no LR-subgroup $Y$ of $X$ with $G < Y < X$. If $X=\Aut(\Gamma)$ we say that $G$ is a maximal LR-group of automorphisms of $\Gamma$.
%
 \end{definition}
%

LR-groups of automorphisms were first introduced in \cite{LR1}  via the notion of an {\em LR-structure}. Although the definition given in \cite{LR1} was for finite graphs only, our definition below works for finite \emph{and} infinite graphs, where in an infinite graph a `cycle' is understood to be a set of edges inducing a connected 2-regular subgraph.

\begin{definition}
\label{def:LR}
 Let $\Gamma$ be a connected tetravalent graph, let $\C$ be
 a partition of $E(\Gamma)$ into cycles, and let $\{\G,\R\}$ be a partition of $\C$ such
 that every vertex of $\Gamma$ is incident to one cycle in $\G$ and one cycle in $\R$.
Define
$$\Aut(\Gamma,\C) = \{ g\in \Aut(\Gamma) \mid \C^g = \C \} \quad \text{and} \quad \Aut^+(\Gamma,\C) = \{g\in \Aut(\Gamma,\C) \mid \G^g = \G \text{ and } \R^g=\R\}.$$
Then the pair $(\Gamma,\C)$ is called an LR-structure, and $\C$ is called an LR-decomposition of $\Gamma$, provided that 
 \begin{enumerate}[leftmargin=*]
 \setlength{\itemsep}{2pt}
  \item[{\rm (a)}] the group $\Aut^+(\Gamma,\C)$ acts transitively on $\V(\Gamma)$, and
  \item[{\rm (b)}] for every $v\in \V(\Gamma)$ and for every cycle $C\in \C$ passing through $v$,
  some $g\in \Aut^+(\Gamma,\C)_v$ acts as a reflection on $C$ and fixes every vertex of the other cycle in $\C$ passing through $v$.
 \end{enumerate}
 An LR-structure $(\Gamma,\C)$ is called {\em self-dual} provided that 
 $\Aut^+(\Gamma,\C)$ is a proper subgroup of  $\Aut(\Gamma,\C)$ and is
  {\em non-self-dual} if $\Aut^+(\Gamma,\C)=
 \Aut(\Gamma,\C)$.
\end{definition}

\begin{remark}
\label{rem:LR2}
{\em 
If there exists a partition 
 $\{\G,\R\}$ of $\C$ satisfying the conditions of the above definition, then by connectedness of $\Gamma$ it is unique, and so $\Aut^+(\Gamma,\C)$ is well-defined. 
 Furthermore, observe that the connectedness of $\Gamma$ implies also that every element
 of $\Aut(\Gamma,\C)$ either preserves each of the sets $\G$ and $\R$ (and hence belongs to $\Aut^+(\Gamma,\C)$), or takes cycles in $\G$ to cycles in $\R$ and vice-versa. In particular, 
 the index of $\Aut^+(\Gamma,\C)$ in $\Aut(\Gamma,\C)$ is at most $2$, and so an LR-structure $(\Gamma,\C)$ is self-dual if and only if there exists some $g \in \Aut(\Gamma,\C)$ such that $\G^g = \R$ and $\R^g = \G$. 
 }
\end{remark}

\begin{remark}
\label{rem:LR}
{\em 
Note that condition (b) of Definition~\ref{def:LR} implies that the  permutation  group
$\Aut^+(\Gamma,\C)_v^{\Gamma(v)}$ 
contains the involutions 
$x:=(w_1\, w_2)$ and $y:=(u_1\, u_2)$
 swapping the neighbours $w_1$ and $w_2$ of $v$ along the unique cycle in $\G$ passing through $v$, 
and swapping the neighbours $u_1$ and $u_2$ of $v$ along the unique cycle in $\R$ passing through $v$.
In particular, $\Aut^+(\Gamma,\C)_v^{\Gamma(v)}$ contains the Klein $4$-group $\langle x, y\rangle$. 
But furthermore, since by definition there is no element in $\Aut^+(\Gamma,\C)$ mapping
a cycle from $\G$ to a cycle to $\R$, we see that $\{w_1,w_2\}$ and $\{u_1,u_2\}$ are orbits
of $\Aut^+(\Gamma,\C)_v^{\Gamma(v)}$, and therefore $\Aut^+(\Gamma,\C)_v^{\Gamma(v)} = \langle x, y\rangle$.
}
\end{remark}

\begin{remark}
\label{rem:LRat}
{\em 
If  $(\Gamma,\C)$ is an LR-structure, and $G:=\Aut^+(\Gamma,\C)$, and $C$ is a cycle in $\C$, then every element $g\in G$ taking an edge of $C$ to an edge of $C$ must preserve the cycle $C$ setwise. Moreover, by vertex-transitivity of $G$ and the existence of automorphisms guaranteed by condition (b) of Definition~\ref{def:LR}, the setwise stabiliser of $C$ in $G$
acts as the full automorphism group of the cycle $C$, and in particular, induces an arc-transitive group on $C$. Then since all cycles in $\G$ lie in the same orbit of $G$, 
as do all the cycles in $\R$, we see that $G$ has precisely two orbits on $A(\Gamma)$: one consists of the arcs of the cycles in $\G$, the other consists of the arcs of the cycles in $\R$.
}
\end{remark}

In Section~\ref{sec:prelim} we will show that the notions of LR-group and LR-structure are equivalent in some sense. 
More precisely, as stated in Lemma~\ref{lem:equiv}, every LR-group of automorphisms $G$
determines a unique LR-structure $(\Gamma,\C)$ for which $G\le \Aut^+(\Gamma,\C)$,
and conversely, if $(\Gamma,\C)$ is an LR-structure then $\Aut^+(\Gamma,\C)$
is a maximal LR-group of automorphisms of $\Gamma$.
\smallskip


There is no obvious reason why one would expect an LR-structure on a given graph to be unique (up to the action of the automorphism group of a graph).
The lack of examples of graphs admitting several non-equivalent LR-structures, however, encouraged the authors of \cite{LR1} to pose the following:

\begin{question} {\rm \cite[Question 1]{LR1}}
\label{q}
If $\C$ and $\C'$ are two distinct LR-decompositions of a finite tetravalent graph $\Gamma$,
 is it true that there exists $g\in \Aut(\Gamma)$ such that $\C^g = \C' {\hskip 1pt}?$ 
\end{question}

For an LR-structure $(\Gamma,\C)$ to be self-dual, it is necessary that the cycles in $\G$ and in $\R$ have the same length. Remarkably, it was shown in \cite[Theorem 8.2]{LR1} that this necessary condition holds provided that $\Aut(\Gamma) \neq \Aut(\Gamma,\C)$. Under the same hypothesis, the authors of \cite{LR1} conjectured that the necessary condition is \emph{also} sufficient:

\begin{conjecture} {\rm \cite[Conjecture 8.1]{LR1}}
\label{c}
If $(\Gamma,\C)$ is a finite LR-structure for which $\Aut^+(\Gamma, \C)$ is a proper subgroup of $\Aut(\Gamma)$, then $(\Gamma,\C)$ is self-dual.
\end{conjecture}

The aim of this paper is to resolve the status of both the  question and the conjecture above.

\begin{theorem}
\label{the:solution}
The answer to Question~\ref{q} is affirmative, and Conjecture~\ref{c} is correct.
\end{theorem}

Our approach to proving Theorem~\ref{the:solution} is based on the following observation. If  $(\Gamma,\C)$ is a non-self-dual LR-structure 
for which $\Aut^+(\Gamma, \C)$ is a proper subgroup of $\Aut(\Gamma)$, then 
there exists a second cycle decomposition $\C'$ for which $(\Gamma,\C')$ is an LR-structure different from $(\Gamma,\C)$;  
indeed $\C^g$ is such a cycle decomposition whenever $g\in \Aut(\Gamma) \setminus \Aut(\Gamma, \C)$. 
 Hence both
Question~\ref{q} and Conjecture~\ref{c} concern the situation where $\Gamma$ admits two distinct LR-structures. 

On the other hand, we prove in Lemma~\ref{lem:S4} that this situation forces the graph $\Gamma$ to be $2$-arc-transitive.  
(Recall that a $2$-arc in a graph is a walk $(u,v,w)$ of length $2$ such that $u \not = w$, and then a vertex-transitive graph $\Gamma$ 
is $2$-arc-transitive if $\Aut(\Gamma)$ acts transitively on the set of $2$-arcs of $\Gamma$, 
which is equivalent to requiring that $\Aut(\Gamma)_v$ acts doubly transitively on $\Gamma(v)$.) 
This allows us to employ structural theory of $2$-arc-transitive groups of automorphisms of tetravalent graphs.
In particular, we will deduce Theorem~\ref{the:solution} in Section~\ref{sec:proof} from the following theorem, which we will prove in Section~\ref{sec:2AT}. 
For this, we note that a group of automorphisms $G$ of a graph $\Gamma$ is called {\em discrete} if the stabiliser $G_v$ is finite 
for every vertex $v\in\V(\Gamma)$; see  \cite{cameron}, for example.

\begin{theorem}
\label{the:main}
Let $A$ be a discrete $2$-arc-transitive group of automorphisms of a connected tetravalent graph $\Gamma$ such that $A_v^{\Gamma(v)} \cong \Sym(4)$. If $A$ contains an LR-subgroup, then $A$ cannot act transitively on the $5$-arcs of $\Gamma$. 
Also, if $G$ is a maximal LR-subgroup of $A$, then there exists an arc-transitive subgroup  $X$ of $A$ containing $G$ as a subgroup of index $2$, 
and every maximal LR-subgroup of $A$ is conjugate to $G$. 
\end{theorem}


In view of the correspondence between LR-structures and maximal LR-groups of automorphisms,
the above theorem has the following corollary.

\begin{corollary}
\label{cor:cor}
If $(\Gamma,\C)$ is an LR-structure such
that $\Aut^+(\Gamma,\C)$ is contained in a discrete arc-transitive group of automorphisms
of $\Gamma$, then $(\Gamma,\C)$ is self-dual.
\end{corollary}

It is not known to us whether Theorem~\ref{the:main} remains valid if the condition on discreteness of the $2$-arc-transitive group $A$ is dropped. 
In any case, as our approach depends heavily on the classification of the discrete $2$-arc-transitive groups of
automorphisms of tetravalent graphs, an entirely different method 
may need
 to be used to analyse the more general situation.

\section{Additional observations}
\label{sec:prelim}

The following lemma is a variation of the well-known fact  that an arc-transitive group $G$
of automorphisms of a connected graph $\Gamma$ is generated by the stabiliser $G_v$ of a vertex $v$ and an element $g$ that reverses an arc incident with $v$. 
Its proof can be derived from the proof of the more general phenomenon of generation of a group of graph automorphisms, as in \cite[Theorem 34]{GenCov};  
but for the sake of completeness, we give an independent proof for our specific context.

\begin{lemma}
\label{lem:generate}
Let $G$ be an LR-group of automorphisms of a connected tetravalent graph $\Gamma$,
let $v$ be a vertex of $\Gamma$, and let $u$ and $w$ be two neighbours of $\Gamma$
belonging to distinct orbits of $G_v$.  If $a$ and $b$ are elements
of $G_{\{v,u\}}{\setminus} G_{vu}$  and $G_{\{v,w\}}{\setminus} G_{vw}$ respectively, 
then $G=\langle G_v, a,b\rangle$.
\end{lemma} 

\begin{proof}
Let $H:=\langle G_v, a,b\rangle$. Observe  that every edge incident to $v$ can be reversed by $a$ or $b$ or one of its $G_v$ conjugates, implying that
$\Gamma(v) \subseteq  v^H = \{v^h : h \in H\}$.
Now suppose $H$ is not transitive on $\V(\Gamma)$.
Then, because $\Gamma$ is connected, there exists an edge $\{x,y\}$ of $\Gamma$ with $x\in v^H,$ say $x = v^h$ where $h \in H$, while $y\not \in v^H.$  
Since $\Gamma(v) \subseteq  v^H,$  it follows that $\Gamma(x) = \Gamma(v^h) = \Gamma(v)^h \subseteq  (v^H)^h = v^H,$  and so $y\in v^H,$ 
a contradiction.  Thus $H$ is transitive on $\V(\Gamma)$, and therefore $G=G_vH = H$.
\end{proof}

\begin{lemma}
\label{lem:S4}
Suppose that $\Gamma$ is a connected tetravalent graph 
admitting two distinct LR-decompositions $\C$ and $\C'$.
Then the group $A:=\langle \Aut^+(\Gamma,\C), \Aut^+(\Gamma,C')\rangle$
acts transitively on the $2$-arcs of $\Gamma$, and
$A_v^{\Gamma(v)} \cong \Sym(4)$
for all $v\in V(\Gamma)$.
\end{lemma}
\begin{proof}
Let $G:=\Aut^+(\Gamma,\C)$ and $H:=\Aut^+(\Gamma,\C')$.
As $\C\not = \C'$, there exists a vertex $v$ of $\Gamma$ and cycles $C\in \C$ and $C'\in \C'$ passing through $v$, sharing precisely one of the edges incident with $v$,  $\{v,u\}$ say, and accordingly, $C(v) = \{u,w\}$ and $C'(v) = \{u,z\}$ for three different neighbours $u,w$ and $z$ of $v$ in~$\Gamma$. 
Now let $x$ be the fourth neighbour of $v$.
By our observations in Remark~\ref{rem:LR}, we know that $G_v^{\Gamma(v)} = \langle (u\, w), (z\, x) \rangle$ and $H_v^{\Gamma(v)} = \langle (u\,z), (w\,x) \rangle$.
Thus $\langle G_v^{\Gamma(v)}, H_v^{\Gamma(v)} \rangle \cong \Sym(4)$, and since 
$A_v^{\Gamma(v)}$ contains $H_v^{\Gamma(v)} $ and $G_v^{\Gamma(v)}$, we find that $A_v^{\Gamma(v)} \cong \Sym(4)$ also.
Finally, since $G$ is vertex-transitive, so is $A$. Hence the $2$-transitivity of $A_v^{\Gamma(v)}$ implies that $A$ is transitive on the set of $2$-arcs of $\Gamma$.
\end{proof}

\begin{lemma}
\label{lem:equiv}
If $G$ is an LR-group of automorphisms of a connected tetravalent graph $\Gamma$, then there exists a unique LR-decomposition $\C$ of $\Gamma$ such that $G\le \Aut^+(\Gamma,\C)$. Conversely, if $(\Gamma,\C)$ is an LR-structure, then $\Aut^+(\Gamma,\C)$ is a maximal LR-group of automorphisms of $\Gamma$. 
\end{lemma}

\begin{proof}
Suppose first that $G$ is an LR-group of automorphisms of $\Gamma$.
By definition, $G$ has two orbits on $\E(\Gamma)$, say $E_1$ and $E_2$. 
So now let $X_1$ and $X_2$ be the graphs with vertex-set $\V(\Gamma)$ and edge-sets $E_1$ and $E_2$, respectively. By connectivity of $\Gamma$, there exists a vertex $v\in\V(\Gamma)$
which is incident to an edge in $E_1$ as well as to one in $E_2$, and since $G$ is vertex-transitive,
it follows that this is true for every vertex of $\Gamma$. Furthermore, since
 $G_v^{\Gamma(v)}$ has two orbits of length $2$, it follows that every vertex of $v$ is incident
 to two edges in $E_1$ and two edges in $E_2$, implying that $X_1$ and $X_2$ are
 both spanning $2$-regular subgraphs of $\Gamma$. 
Let $\G$, respectively, $\R$, be the set consisting of the sets of edges of the cycles in $X_1$, respectively, $X_2$, and let $\C = \G \cup \R$.
 Then $\C$ is clearly a $G$-invariant decomposition of $\E(\Gamma)$
 into cycles, with each vertex of $\Gamma$ incident to one cycle in $\G$ and to one cycle
 in $\R$. Moreover, since $E_1$ and $E_2$ are orbits of $G$ on $\E(\Gamma)$, we see
 that $G$ is a subgroup of $\Aut^+(\Gamma,\C)$, with respect to the partition $\{\G,\R\}$ of~$\C$. 
In  particular, $\Aut^+(\Gamma,\C)$ acts transitively on $\V(\Gamma)$.
 
Next, to show that condition (2) of Definition~\ref{def:LR} holds, consider the cycles
$C\in \G$ and $D\in \R$ passing through a vertex $v$, 
and let $C(v) =  \{u_1, w_1\}$ and $D(v) =\{u_2, w_2\}$ denote the neighbourhoods of $v$ in these two cycles.
By the construction of $\G$ and $\R$, we see that $G_v^{\Gamma(v)}$ preserves
each of  $C(v)$ and $D(v)$ setwise, which together with the requirement that 
$G_v^{\Gamma(v)} \cong \V_4$ implies that $G_v^{\Gamma(v)} = \langle (u_1\, w_1), (u_2\, w_2) \rangle$. 
Now observe that an element of $G_v$ inducing the permutation $(u_1\, w_1)$ on $\Gamma(v)$ preserves
both $C$ and $D$ setwise, and moreover, as it fixes three consecutive vertices on the cycle
$D$, it fixes $D$ point-wise. Similarly, since this element fixes $v$ but swaps the two $C$-neighbours
of $v$, it reflects $C$ at $v$. By applying an analogous argument to the permutation
$ (u_2\, w_2)$, we see that the condition (2) of Definition~\ref{def:LR} is indeed fulfilled,
completing the proof that $(\Gamma,\C)$ is an LR-structure with $G\le \Aut^+(\Gamma,\C)$.


Suppose now that  $\C$ and $\C'$ are LR-decompositions of $\Gamma$ for which 
$G\le \Aut^+(\Gamma,\C)$ and $G \le  \Aut^+(\Gamma,C')$. If $\C \neq \C'$, then there is some vertex $v \in V(\Gamma)$ and some $C\in \C$, $C' \in \C'$ such that for $x,y,z\in \Gamma(v)$ we have $x,v,y \in C$ and $x,v,z \in C'$. Since $G \le \Aut^+(\Gamma,\C)$, the set $\{x,y\}$ is an orbit of $G_v$. On the other hand, since  $G \le \Aut^+(\Gamma,\C')$, the set $\{x,z\}$ is an orbit of $G_v$. This is a contradiction, and hence   $\C=\C'$.

Conversely, let $(\Gamma,\C)$ be an LR-structure and let $G=\Aut^+(\Gamma,\C)$. 
Then $G_v^{\Gamma(v)}$ is permutation isomorphic to the Klein $4$-group $\V_4$ in its intransitive action on $4$ points, by our observations in Remark~\ref{rem:LR}. On the other hand, $G$ is vertex-transitive (by definition), but $G$ is not edge-transitive as it preserves the sets $\G$ and $\R$. 
Also vertex-transitivity and the fact that $G_v$ induces $\V_4$ on $\Gamma(v)$ 
imply that $G$ is transitive on the set of edges contained in the cycles in $\G$, as well as
on the set of  edges contained in $\R$. In particular, $G$ has two orbits on $\E(\Gamma)$.
This shows that $G$ is an LR-group of automorphisms of $\Gamma$. 
Moreover, if $X$ is another LR-group of automorphisms of $\Gamma$ such that $G\le X$,
then from the first paragraph of this proof we know there exists a unique LR-decomposition $\C'$ of $\Gamma$ for which $X\le \Aut^+(\Gamma,\C')$. 
But then $G\le \Aut^+(\Gamma,\C')$, and again
by the uniqueness of the LR-decomposition $\C$ for which $G\le \Aut^+(\Gamma,\C)$, we find 
that $\C = \C'$. It follows that $X\le \Aut^+(\Gamma,\C) = G$, and hence $X=G$. This shows that
$G$ is a maximal LR-group of automorphisms of $\Gamma$, and completes the proof.
\end{proof}

\begin{corollary}
\label{cor:max}
If $(\Gamma,\C)$ is an LR-structure, and $\Aut^+(\Gamma,\C) < X \le \Aut(\Gamma)$, then $X$ acts transitively on the arcs of $\Gamma$.
\end{corollary}

\begin{proof}
By Lemma~\ref{lem:equiv}, we see that $G:=\Aut^+(\Gamma,\C)$ is a maximal LR-group of automorphisms of $\Gamma$, and so $X$ is not an LR-group of automorphisms.
Now suppose that $X$ is not transitive on $\A(\Gamma)$. Then $X_v^{\Gamma(v)}$ is not transitive, and so $X_v^{\Gamma(v)} = G_v^{\Gamma(v)} \cong \V_4$, and since $X$ is not an LR-group, $X$ must be transitive on the edges of $\Gamma$. In view of Remark~\ref{rem:LRat}, it follows that $G$ acts transitively on the arcs underlying each of the two edge-orbits of $G$. Since these two $G$-edge-orbits are merged into a single $X$-edge-orbit, this implies that $X$ is arc-transitive on $\Gamma$ after all, a contradiction. 
\end{proof}

\begin{lemma}
\label{lem:index2}
Let $(\Gamma,\C)$ be an LR-structure, and let  $G:=\Aut^+(\Gamma,\C)$.
Then the following claims are equivalent$\,:$
\begin{itemize}
\item[{\rm (i)}]
$(\Gamma,\C)$ is self-dual$\,;$
\item[{\rm (ii)}]
there exists a group $X$ such that $G\le X\le \Aut(\Gamma)$ and $|X:G| = 2\,;$ 
\item[{\rm (iii)}]
there exists an arc-transitive but not $2$-arc-transitive group $X \le \Aut(\Gamma)$ that contains $G$ as a normal subgroup.
\end{itemize}
\end{lemma}

\begin{proof}
Suppose first that (i) holds.
In this case, let $\{\G,\R\}$ be the partition of $\C$ as in Definition~\ref{def:LR}, and let $X=\Aut(\Gamma,\C)$. 
Then clearly the partition $\{\G,\R\}$ is $X$-invariant, and $G$ is the kernel of the induced action of $X$ on $\{\G,\R\}$,  
so $|X:G| = 2$, and hence (ii) holds.

Next, suppose that (ii) holds. Then by Corollary~\ref{cor:max}, $X$ is arc-transitive.  Moreover, as  both $X$ and $G$ are vertex-transitive, 
we find that $|X_v:G_v| = |X:G| = 2$, which implies that the $G_v$-orbits on $\Gamma(v)$ are blocks of imprimitivity for $X_v^{\Gamma(v)}$. 
In particular, $X_v^{\Gamma(v)}$ cannot be doubly transitive, and so $X$ is not $2$-arc-transitive. This proves that (iii) holds.

Finally, suppose that (iii) holds.
Since $G$ is normal in $X$ and since both $\G$ and $\R$ are orbits of the action of $G$ on $\E(\Gamma)$,
 the partition $\{\G,\R\}$ is $X$-invariant.
Since $X$ is arc-transitive, this implies that there exists $g\in X$ interchanging $\G$ and $\R$, and therefore $\C$ is self-dual, that is, (i) holds.
\end{proof}

\section{Amalgams and a reduction to trees}
\label{sec:tree}

We begin this section by recalling some basic facts about finite group amalgams of rank $2$
and their relationship with discrete groups acting arc-transitively on graphs.

Let $\Gamma$ be a connected $d$-valent graph, and let $G$ be a discrete arc-transitive group of automorphisms of~$\Gamma$. 
Also let $u$ and $v$ be two adjacent vertices in $\Gamma$, 
 let $L=G_v$, $R=G_{\{u,v \}}$, $B= G_{uv}$ be the stabilisers of the vertex $v$, edge $\{u,v\}$
and arc $(u,v)$, respectively. 
Then it is well known that the following hold (see, \cite{Pot} for example):
\begin{enumerate}
 \setlength{\itemsep}{2pt}
\item[(1)] $B=L \cap R$ is a finite group, 
\item[(2)] $|L:B| = d$ and $|R:B| = 2$, 
\item[(3)] if $K \le B$ and $K$ is normal in $L$ and in $R$, then $K=1$, and 
\item[(4)] $G=\langle L, R\rangle$. 
\end{enumerate}

Furthermore, if $a$ is an arbitrary element of $R {\setminus} B$ (so that $a$ reverses the arc $(u,v)$), 
then $\Gamma$ is isomorphic to the {\em Schreier coset graph} $\Cos(G;L,a)$ whose vertices are
the right cosets of $L$ in $G$ and whose edges are of the form $\{Lx,Lax\}$ for $x \in G$.  
This graph can also be denoted by $\Cos(G;L,B,R)$, with arcs being the cosets of $B$ in $G$
(with the initial vertex of an arc $Bx$ being $Lx$, and the reverse of $Bx$ being $Bax$).  
Moreover, via this isomorphism between $\Gamma$ and $\Cos(G;L,B,R)$, the action of $G$ on the 
right cosets of $L$ and $B$ by right multiplication corresponds to the original action of $G$ on the vertices and arcs of $\Gamma$.

Conversely, suppose that $B, L$, $R$ and $G$ are arbitrary groups satisfying conditions (1) to (4) above. 
In this case we say that $(L,B,R)$ is a {\em finite faithful amalgam of index $(d,2)$}, and that $G$ is a {\em completion} of the amalgam.
Correspondingly, $\Gamma = \Cos(G;L,B,R)$ is a connected regular $d$-valent graph, upon which $G$ acts 
by right multiplication as a group of automorphisms.  Moreover, this action is faithful, and transitive on the arcs of $\Gamma$. 
Note that the groups $L$, $R$ and $B$ play the roles of the vertex-stabiliser, edge-stabiliser and the arc-stabiliser of
a mutually incident vertex-edge-arc triple in $\Gamma$.
\smallskip

For a given finite faithful amalgam $(L,B,R)$, there exists a {\em universal completion} $\tG$, denoted by 
$L*_B R$ and called the free product of $L$ and $R$ amalgamated over $B$.  This has the property that for every
completion $G$ of the given amalgam, there exists an epimorphism $\pi \colon \tG \to G$ whose kernel intersects the groups $L$ and $R$
trivially (so that we may identify the subgroups $L, R,B \le \tG$ with their $\pi$-images in $G$). 
Furthermore, $\tGa:= \Cos(\tG;L,B,R)$ is an infinite $d$-valent graph, called {\em the universal cover of $\Gamma$}, 
with $\pi$ induces a covering projection $\tGa \to \Gamma$.

Let us now focus on the situation arising in Theorem~\ref{the:main} and prove the following statement:
%

%
\begin{lemma}
\label{eq:1}
Theorem~\ref{the:main} holds provided that it holds in the case where $\Gamma$ is a tetravalent tree.
\end{lemma}

\begin{proof}
Suppose Theorem~\ref{the:main} holds for the case in which the graph in question is a tetravalent tree.  

Now let $\Gamma$ be any connected tetravalent graph admitting a discrete $2$-arc-transitive group $A$ of automorphisms 
with $A_v^{\Gamma(v)} \cong \Sym(4)$, and $A$ contains an LR-subgroup $G$.
Also let $\{u,v\}$ be an edge of $\Gamma$, and let $L:=A_v$, $R:=A_{\{u,v \}}$ and $B:= A_{uv}$. 
Then we may identify $\Gamma$ with $\Cos(A;L,B,R)$, and the action of $A$ on the vertices,
 arcs and edges of $\Gamma$ with the action of $A$ on cosets of $L$, $B$ and $R$, respectively.

Next, let $\tA:=L*_B R$ be the universal completion of the amalgam $(L,B,R)$,
let $\pi\colon \tA \to A$ be the corresponding epimorphism, and let
$\T_4:=\Cos(\tA;L,B,R)$. Then $\T_4$ is a $4$-valent tree upon which $\tA$ acts arc-transitively. Moreover, since the groups $L$, $B$ and $R$
are the stabiliser of an incident vertex, arc and edge (respectively), we see that in the actions of the group $\tA$ on $\T_4=\Cos(\tA;L,B,R)$) 
and the group $A$ on $\Gamma=\Cos(A;L,B,R)$), the vertex-stabiliser $\tA_\tv = L$ acts on
the neighbourhood $\tGa(\tv)$ in the same way as the vertex-stabiliser $A_v = L$ on the neighbourhood $\Gamma(v)$.
Thus $\tA_\tv^{\T_4(\tv)}$ is permutation isomorphic to $A_v^{\Gamma(v)} \cong \Sym(4)$. 
Also, because $A$ is $2$-arc-transitive and discrete, so is $\tA$, and $\tA_\tv^{\T_4(\tv)} \cong \Sym(4)$. 

Moreover, using the fact that $\pi$ induces a covering projection $\T_4 \to \Gamma$, one can also show that
$A$ acts $s$-arc transitively on $\Gamma$ if and only if $\tA$ acts $s$-arc transitively on $\T_4$. (This fact is well-known; for a proof see \cite[Lemma 3.2]{GiudiciSwartz}.)

Now let $\tG$ be the $\pi$-preimage of the LR-subgroup $G$ of $A$, and observe that the stabilisers 
of some incident vertex, arc and edge of $\tG$ are equal to the intersections $\tG\cap L$, $\tG\cap B$ and $\tG\cap R$, respectively.
By an argument similar to the one above, we see that the action of $\tG_\tv$ on
 $\T_4(\tv)$ is isomorphic to the action of 
$G_v$ on $\Gamma(v)$, and therefore $\tG_\tv^{\T_4(\tv)}$ is permutation isomorphic
to $G_v^{\Gamma(v)}\cong \V_4$. Then since the action of $\tG$ on the edge-set
 $\E(\T_4)$ is isomorphic to the action of $\tG$ on right cosets of the copy of $R$ in $\tA$ by right multiplication,
 which in turn is isomorphic to the action of $G$ on right cosets of $R$ in $A$ by right multiplication,
 we see that the number of edge-orbits of $\tG$ on $\T_4$ is equal to the number of 
 edge-orbits of $G$ on $\Gamma$. In particular, since $G$ is an LR-group of automorphisms, 
 it follows that so is $\tG$.
  
 Finally, observe that the epimorphism $\pi$ induces a bijection between the lattice of subgroups of $\tA$ that contain $\ker(\pi)$ and the lattice of subgroups of $A$. It follows that
 there exists a subgroup $\tX$ of $\tA$ containing $\tG$ as a subgroup of index $2$ if and only if
there exists a subgroup $X$ of $A$ containing $G$ as a subgroup of index $2$. 
Similarly, $\tG$ is a maximal LR-subgroup of $\tA$ if and only if $G$ is a maximal LR-subgroup of $A$.
 
 We may now use our assumption that Theorem~\ref{the:main} holds for trees
 to conclude that $\tA$ does not act $5$-arc-transitively on $\T_4$,
 and hence that $A$ does not act $5$-arc-transitively on $\Gamma$. 
Also if $G$ is a maximal LR-subgroup of $\Gamma$, then $\tG$ is 
 a maximal LR-subgroup of $\tA$, and hence by our assumption, there exists
 an arc-transitive subgroup $\tX$ of $\tA$ containing $\tG$ as a subgroup of index $2$.
Thus $X:=\pi(\tX)$ is an arc-transitive subgroup of $A$ containing $G$ as a subgroup of index $2$. 
Moreover, if $\Gamma$ admits another maximal LR-subgroup $H$ of $A$,
  then by the same argument as above, $\pi^{-1}(H)$ is a maximal LR-subgroup of $\tA$,
  and hence (by our assumption) is conjugate in $\tA$ to $\tG$. But then $H$ is conjugate
  to $G$ in $A$, and this completes the proof.
\end{proof}

\section{LR-subgroups of $2$-arc-transitive groups}
\label{sec:2AT}

The structure of discrete $2$-arc-transitive groups of automorphisms of the infinite tetravalent tree $\T_4$ is well understood, thanks to the classical work of Gardiner \cite{Gar} and Weiss \cite{weiss}, and a description of these groups in terms of generators and relators was given explicitly in \cite[Table 1]{Pot}.
It follows from this work that up to conjugacy in $\Aut(\T_4)$, there are exactly nine possibilities for the group $A$, with six of those satisfying the additional condition
$A_v^{\T_4(v)} \cong \Sym(4)$. 

These six groups, together with the corresponding stabilisers $(L,B,R)$ of mutually incident vertex-arc-edge triples in the graph,
are given in the second column of Table~\ref{tab:1}. Moreover, presentations for the subgroups $L$, $B$ and $R$ can be read conveniently from the presentation of the corresponding group $A$ in Table~\ref{tab:1}, by simply taking all the relators that involve only the generators of the subgroup in question.

Note that the names of the first four of the groups in Table~\ref{tab:1} are chosen so as to reflect
the isomorphism type of the vertex-stabiliser. For example, the group $A$ named $S_4$ has
vertex-stabiliser $L$ isomorphic to $\Sym(4)$, while for the groups named $C_3\rtimes S_4$ and
$C_3\rtimes S_4^*$, the vertex-stabiliser $L$ is isomorphic to a semidirect product $C_3\rtimes \Sym(4)$, 
and for the group named $S_3\times S_4$ the vertex-stabiliser is isomorphic to the direct product  $\Sym(3)\times \Sym(4)$. 
Each of these four possibilities for the group $A$ act on $\T_4$ either $2$-arc-transitively or $3$-arc-transitively (but not $4$-arc-transitively). The last two groups, named 4-AT
and 7-AT, act $4$-arc-transitively (but not $5$-arc-transitively) and $7$-arc-transitively, respectively.

\begin{table}[!htbp]
\begin{scriptsize}
\begin{tabular}{|c|c|c|}
\hline 
Name & $A=L*_B R$ & \begin{tabular}{c} \\ maximal LR-subgroup $G$\\ \\ \hline \\ $\hbox{N}_A(G)$ \\ \\\end{tabular}  \\
\hline\hline
$S_4$ & \begin{tabular}{c} \\
           $A = \la x,y,s,t,a \mid x^2, y^2, s^3, t^2, a^2, [x,y], s^t=s^{-1}, x^s=y,  y^s=xy, x^t=y, [s,a], [t,a]\ra$  \\ \\
           $L=\la x,y,s,t \ra ,\quad B=\la s,t \ra, \quad  R=\la s,t,a\ra$ \\ \\
         \end{tabular} 
& \begin{tabular}{c} $\langle xy,t,a,a^x\rangle$  \\ \\ \hline \\ $\langle x,y,t,a\rangle$ \end{tabular}
         \\
\hline
$C_3\rtimes S_4$ & \begin{tabular}{c} \\
           $A = \la x,y,c,d,t,a \mid x^2,y^2,c^3,d^3,t^2,a^2, [x,y], [c,d]$,\\
            $[c,x], [c,y], (tc)^2, (td)^2, x^d=y, y^d=xy, x^t=y,
            c^a=d, [a,t] \ra$  \\ \\
           $L=\la x,y,c,d,t \ra,\quad B=\la c,d,t \ra,\quad R=\la c,d,t,a\ra $
           \\ \\
         \end{tabular} 
& \begin{tabular}{c} $\langle xy,t,a,a^x\rangle$  \\ \\ \hline \\ $\langle x,y,t,a\rangle$ \end{tabular}
         \\
\hline
$C_3\rtimes S_4^*$ & \begin{tabular}{c} \\
           $A = \la x,y,c,d,t,a \mid x^2,y^2,c^3,d^3,t^2,a^2=t, [x,y], [c,d],$\\
            $[c,x], [c,y], (tc)^2, (td)^2, x^d=y, y^d=xy, x^t=y,
            c^a=d, d^a = c^{-1} \ra$  \\ \\
           $L=\la x,y,c,d,t \ra,\quad B=\la c,d,t \ra,\quad R=\la c,d,t,a\ra$ \\ \\
         \end{tabular} 
& \begin{tabular}{c} $\langle xy,t,a,a^x\rangle$  \\ \\ \hline \\ $\langle x,y,t,a\rangle$ \end{tabular}
         \\
 \hline 
 $S_3\times S_4$ & \begin{tabular}{c} \\
           $A = \la x,y,c,d,r,s,a \mid x^2,y^2,c^3,d^3,r^2,s^2,a^2, [x,y], [c,d], [r,s], [c,x], [c,y]$,\\
            $c^r = c^{-1}, [d,r], [c,s], d^s=d^{-1}, x^d=y, y^d=xy, x^s=y, [r,x], [r,y], c^a=d, s^a =r  \ra$  \\ \\
           $L=\la x,y,c,d,r,s \ra,\quad B=\la c,d,r,s \ra, \quad R=\la c,d,r,s,a\ra$ \\ \\
         \end{tabular}
& \begin{tabular}{c} $\langle xy, s,r,a,a^x\rangle$  \\ \\ \hline \\ $\langle x,y, s,r,a\rangle$ \end{tabular}
         \\
 \hline

 $4$-AT & \begin{tabular}{c} \\
           $A = \la t,c,d,e,x,y,a \mid t^2, c^3, d^3, e^3, x^2, y^2, a^2, [c,d], [c,e], [d,e] =c, [x,y], (cx)^2$,\\
                  $(dx)^2, [e,x], (cy)^2, [d,y],(ey)^2, c^t=d^{-1},$\\
                  $y(et)^2e^{-1}te^{-1}, (et)^4x, (ca)^2, d^a=e, x^a=y  \ra$
            \\ \\
           $L=\la t, x,y,c,d,e \ra,\quad B=\la x,y,c,d,e \ra,\quad R=\la x,y,c,d,e,a \ra$ \\ \\
         \end{tabular}
& \begin{tabular}{c} $\langle t, x, y, a, (ca)^{(cet)^{-1}}  \rangle$  \\ \\ \hline \\ 
   $\langle t, x, y, ete, a, (ca)^{(cet)^{-1}} \rangle$ 
 \end{tabular}
         \\
 \hline 
 $7$-AT & \begin{tabular}{c} \\
           $A = \la h,p,q,r,s,t,b,c,k,a  \mid h^4, p^3, q^3, r^3, s^3, t^3, b^3, c^2, k^2, a^2, kh^2,$\\
                  $[p,q], [p,r], [p,s], [p,t], [p,b],  [q,r], [q,s], [q,t],[q,b],$ \\
                  $[r,s], [r,t], [b,s], [s,t]=p, [b,r]=q, [t,b]=qrsp^{-1}$,\\
                  $[k,c], (tk)^2,(rk)^2, [p,k],(qk)^2,(sk)^2, [b,k],(tc)^2, [r,c], (pc)^2, (qc)^2, [s,c], (bc)^2$, \\
                  $[p,h], q^h=q^{-1}r, r^h=qr, s^h=pq^{-1}r^{-1}s^{-1}t^{-1}, t^h=p^{-1}qr^{-1}s^{-1}t,$\\
                   $(hbc)^2, (hb)^3, p^a=q^{-1}, r^a =s^{-1}, t^a=b^{-1}, [c,a], k^a=ck \ra$  \\ \\
           $L=\la h,p,q,r,s,t,b,c,k \ra,\quad B=\la p,q,r,s,t,b,c,k \ra,\quad R=\la p,q,r,s,t,b,c,k,a\ra$. \\ \\
         \end{tabular}
         & NONE             
         \\
 \hline        
\end{tabular}
${}$\\[6pt] 
\caption{Discrete $2$-arc-transitive groups $A\le \Aut(\T_4)$ with {$A_v^{\T_4(v)} \cong \Sym(4)$}  and their maximal LR-subgroups}
\label{tab:1}
\end{scriptsize}
\end{table}

We will now prove that Theorem~\ref{the:main} holds for the case where
$\Gamma$ is the tetravalent tree $\T_4$, which by Lemma~\ref{eq:1},
then proves Theorem~\ref{the:main} in full generality. In fact, we prove something slightly stronger: 

 \begin{lemma}
 \label{lem:detail}
 Let $A$ be a discrete $2$-arc-transitive group of automorphisms of $\T_4$.
 If $A$ is one of the groups in the first five rows of Table~\ref{tab:1}
 (that is, if $A$ has type $S_4$, $C_3\rtimes S_4$, $C_3\rtimes S_4^*$,  $S_3\times S_4$ 
 or 4-AT), then $A$ contains a maximal LR-subgroup $G$, unique up to conjugation in $A$, 
 given in the third column of Table~\ref{tab:1}. Also the normaliser
  $\N_A(G)$ in $A$ of the maximal LR-subgroup $G$ contains $G$ as a subgroup of index $2$, 
  and is shown in the third column of Table~\ref{tab:1} too. 
Finally, if $A$ has type {\em 7-AT} (as in the sixth row), then $A$ contains no LR-subgroup.
 \end{lemma}

 
\begin{proof}
Let $\{v,u\}$ be an edge of $\Gamma:=\T_4$, and recall that
 we may assume that 
$(A,A_v,A_{vu},A_{\{v,u\}})$ is one of the quadruples $(A,L,B,R)$ given in Table~\ref{tab:1}, 
and that the vertex-set, arc-set and edge-set of $\Gamma$ can be identified with the right coset spaces
$(A\!:\!A_v)$, $(A\!:\!A_{vu})$ and $(A\!:\!A_{\{v,u\}})$, respectively, 
with the action of $A$ on vertices, arcs and edges of $\Gamma$ coinciding with the actions of $A$ on $(A\!:\!A_v)$, $(A\!:\!A_{vu})$ and $(A\!:\!A_{\{v,u\}})$ 
by right multiplication. 
We proceed using a combination of theoretical and computational methods. 
\smallskip

%
 %

Let $\Omega$ be the right coset space $(A_v\!:\!A_{vu})$, 
let $\rho \colon A_v \to \Sym(\Omega)$ be the natural action of $A_v$ on $\Omega$, 
and let $A_v^\Omega$ be the permutation group $\rho(A_v)$ induced by this action on $\Omega$.
Since $A_{vu}$ is the stabiliser of $u$ in the transitive action of $A_v$ on $\Gamma(v)$,
we may identify the elements of $\Gamma(v)$ with the elements of $\Omega$
in such a way that the coset action $\rho$ of $A_v$ on $\Omega$
corresponds to the action of $A_v$ on $\Gamma(v)$, and  that $A_v^{\Gamma(v)}$
corresponds to $A_v^\Omega$. In particular, $A_v^\Omega=\Sym(\Omega)$.

By the definition of  an LR-group, $G_v$ is a $2$-subgroup of $A_v$ such that $G_v^{\Gamma(v)}$ is permutation
isomorphic to the intransitive Klein $4$-group $\V_4$. With the above-described
identification of $\Gamma(v)$ with $\Omega$, we see that $G_v^\Omega:= \rho(G_v)$ is one of the three intransitive subgroups 
of $\Sym(\Omega)$ isomorphic to $\V_4$, and hence that $G_v$ belongs to the set
$\X$ of all $2$-subgroups of $A_v$ with $\rho(X) \cong \V_4$. 

Because $A_v$ is a finite group (with the presentation given in Table~\ref{tab:1}), one can 
use a computer algebra system such as {\sc Magma} \cite{magma} to determine the set $\X$ for each of the six possible types of the group $A$. 

If $A$ has type $S_4$, $C_3\rtimes S_4$ or
$C_3\rtimes S_4^*$, then the set $\X$ consists of a single conjugacy class of subgroups of $A_v$, 
with the class representative being the Klein $4$-subgroup generated by $xy$ and $t$.
If $A$ has type $S_3\times S_4$, then $\X$ is the disjoint union of four $A_v$-conjugacy classes,
the representatives of which are 
$\langle x y, s \rangle$,
$\langle x y, sr \rangle$,
$\langle rx y, s \rangle$ and
$\langle x y, r,s \rangle$, 
with the first three isomorphic to $\V_4$, and the fourth being an elementary abelian group of order $8$.
If $A$ has type 4-AT, then $\X$ consists of a unique $A_v$-conjugacy class represented by
the group $\langle x, y,  t \rangle$, isomorphic to the dihedral group ${\rm D}_4$ of order $8$.
Finally, if $A$ is of type 7-AT, then $\X$ is the conjugacy class in $A_v$ of 
$ \langle pcq, (pcq)^h\rangle$, 
which is also isomorphic to ${\rm D}_4$.

Next, since $G$ acts transitively on $\V(\Gamma)$, we see that $A =G A_v$, and so 
$G$ has finite index in $A$, indeed $|A : G| = |A_v : G_v|$.
Hence the group $G$ is a member of the set
$$\cT = \bigcup_{X\in \X} \{\, T \,: \,X \le T\le A \,\hbox{ with } \,|A:T| = |A_v:X|\,\}.$$
We computed this set $\cT$ for each of the first five types of the group $A$, 
using the {\tt LowIndexSubgroups} routine in {\sc Magma}. 
(The computation in these cases takes only a few seconds on an average laptop.) 
We were unable to do the same in the case where $A$ has type 7-AT, however, due to the computational complexity of the 
{\tt LowIndexSubgroups} algorithm, we will explain how we dealt with this case later.


Of course the set $\cT$ might contain some subgroups of $A$ that are
not vertex-transitive, or are vertex-transitive but act half-arc-transitively on $\Gamma$, 
and so we can restrict our attention to the subset $\cT^*$ consisting of all $T \in \cT$ that are LR-groups of automorphisms of $\Gamma$. 

In determining this set $\cT^*$, we observe that an LR-group $T\in \cT^*$ has two orbits on $\E(\Gamma)$ and $\A(\Gamma)$,
and hence that each of $A_{\{v,u\}}$ and $A_{vu}$ has two orbits in its action on the coset space $(A\!:\!T)$ by right multiplication. 
Moreover, as an LR-group $T$ is vertex-transitive, we see that $A_v$ is transitive in its action on $(A\!:\!T)$, and hence a group  $T\in \cT$ belongs to $\cT^*$ if and only if 
 the stabilisers $A_{\{v,u\}}$, $A_{vu}$ and $A_v$ have two orbits, two orbits and one orbit, respectively, in the their actions on $(A\!:\!T)$. 

Next, because  every group $T\in \cT$   has finite (and relatively small) index in $A$, it is easy check each group $T \in \cT$  against the latter 
 condition computationally, and find that, up to conjugacy in $A$, the set $\cT^*$ consists of a unique
 LR-group in all cases, except when $A$ has type $S_3\times S_4$ or possibly when it has type 7-AT (in which case the above approach is computationally too difficult). 
 
In all but those two exceptional cases, the LR-groups that satisfy the condition are precisely the groups listed in the third column of Table~\ref{tab:1}.
 If $A$ has type $S_3\times S_4$, then $\cT^*$ consists of two conjugacy classes, with representatives $\langle xy, sr, a, a^x\rangle$ and  $\langle xy, a, a^x, s,r\rangle$, 
 and as the former is clearly a subgroup of the latter, we find that (again) $A$ contains a unique maximal LR-subgroup up to conjugacy. 
  
Also, a direct computation shows that in each case (except possibly when $A$ is of type 7-AT), the normaliser $\N_A(G)$ in the unique maximal LR-subgroup $G$ of $A$ is
  as stated in Table~\ref{tab:1} and contains $G$ as a subgroup of index $2$. This proves the
  statement of the lemma in all cases except when $A$ is of type 7-AT.
\\[-8pt]

Finally, let us assume that $A$ is of type 7-AT, and hence that
 %
 $G_v =  \langle \alpha, \alpha^h\rangle \cong {\rm D}_4$ where $\alpha=pcq$. 
Recall that $\rho$ is the action of $A_v$ on the set of right cosets of $A_{vu}$ ($=B$), and that  $\alpha$ fixes $u$. Let $x$ be the other vertex in $\Gamma(v)$  fixed by $\alpha$, and let $w$ and $z$ be vertices in $\Gamma(v)$ that are interchanged by $\alpha$. 
Now since $|G_v|=|G_{vu}|\,|u^{G_v}| = 2 |G_{vu}|$, we find that $|G_{vu}| = 4$.
Also we see that $\rho(h)$ is the double transposition $(u\,w)(x\,z)$, and
hence $\alpha\in G_{vu}$ and $\alpha^h\in G_{vw}$.
Moreover,  one can see that the element $(\alpha^h\alpha)^2$  of order $2$
lies in the kernel of 
 $\rho$,
  and it follows 
that $(\alpha^h\alpha)^2$ fixes the neighbourhood $\Gamma(v)$ pointwise, 
and therefore $\alpha \not = (\alpha^h\alpha)^2 \not = \alpha^h$. 
Also because $|G_{vu}| = |G_{vw}| = 4$, we find  that
$G_{vu} =  \langle (\alpha^h\alpha)^2,\alpha\rangle$
and $G_{vw} = \langle (\alpha^h\alpha)^2,\alpha^h\rangle$.

We will now determine the edge-stabiliser $G_{\{v,u\}}$. To do this, we first observe 
that $(\alpha^h\alpha)^2 = q^2 k$, which can be verified easily since it involves only elements in the finite group $G_v$, 
and from this we find that $G_{vu} = \langle pcq, q^2k\rangle$. Then since the elements $pcq$ and $q^2k$ 
involve only generators of the group $R = A_{\{v,u\}}$, we deduce that the edge-stabiliser $G_{\{v,u\}}$ is precisely the
normaliser $\N_{A_{\{v,u\}}}(G_{vu})$, and this gives $G_{\{v,u \}} = \langle G_{vu}, a\rangle \cong {\rm D}_4$.


Recall that $\Gamma(v)= \{u,x,w,z\}$ and that $\rho(G_v)=\langle  (w\,z) , (u\,x) \rangle$. This means $u$ and $w$ are in distinct  orbits of $G_v$ on $\Gamma(v)$. Thus, by Lemma~\ref{lem:generate}, there exist $\mu \in G_{\{v,u\}} \setminus G_{vu}$ 
 and $\nu\in G_{\{v,w\}} \setminus G_{vw}$ such that $G = \langle G_v, \mu, \nu\rangle$. Now $\mu= \mu' a$ for some $\mu' \in G_{vu}$ and since $G_{\{ v,w\}}=(G_{\{v,u\}})^h$,  also $\nu = \nu' a^h$ for some $\nu' \in G_{vw}$. Hence $G = \langle G_v, a, a^h \rangle = \langle (pcq), (pcq)^h, a, a^h \rangle$. A quick
  computation in {\sc Magma} shows that $G=A$, which is a contradiction to our assumption that $G$ is an LR-group. This shows  that if $A$ is of type 7-AT (or equivalently, is $7$-arc-transitive), then $A$ contains no LR-groups.
\end{proof}

\section{Proof of Theorem~\ref{the:solution} and Corollary~\ref{cor:cor}}
\label{sec:proof}

We conclude this paper by deducing Theorem~\ref{the:solution} from what we found 
in the previous sections and then proving Corollary~\ref{cor:cor}.
\smallskip

Let us begin by answering Question~\ref{q}, namely that if $\C$ and $\C'$ are distinct LR-decompositions of a finite tetravalent graph $\Gamma$,
does there exist $g\in \Aut(\Gamma)$ such that $\C^g = \C' {\hskip 1pt}?$ 

Let $\Gamma$, $\C$ and $\C'$ be as above and let $G:=\Aut^+(\Gamma,\C)$ and $G^*:=\Aut^+(\Gamma,\C')$. Then, by Lemma~\ref{lem:S4},
the group $A:=\langle G, G^*\rangle$ acts $2$-arc-transitively on $\Gamma$ with $A_v^{\Gamma(v)} \cong \Sym(4)$. Moreover,
by Lemma~\ref{lem:equiv}, the groups $G$ and $G^*$ are maximal LR-subgroups of $A$.
By Theorem~\ref{the:main},  $G$ and $G^*$ are conjugate within $A$, so 
$G^* = G^g$ for some $g\in A$. It follows that $\C^g$ is an LR-decomposition of $\Gamma$
with $G^* = \Aut^*(\Gamma,\C^g)$, and now Lemma~\ref{lem:equiv} implies that $\C^g=\C'$.
Hence the answer to Question~\ref{q} is affirmative. 
\smallskip

Finally, we verify Conjecture~\ref{c}, namely that if $(\Gamma,\C)$ is a finite LR-structure for which $\Aut^+(\Gamma, \C)$ is a proper subgroup of $\Aut(\Gamma)$, then $(\Gamma,\C)$ is self-dual.

Suppose that $(\Gamma,\C)$ is a counterexample to Conjecture~\ref{c}.
  Then
 $\Aut(\Gamma,\C) = \Aut^+(\Gamma,\C) = G$, and 
  there exists an element
 $h\in \Aut(\Gamma) \setminus G$. Now $\C^h$ is an LR-decomposition
 of $\Gamma$ distinct from $\C$, and by Lemmas~\ref{lem:S4} and ~\ref{lem:equiv},
 the group $A:=\langle G, G^h\rangle$ acts $2$-arc-transitively on $\Gamma$ with $A_v^{\Gamma(v)}\cong \Sym(4)$ and with $G$ being a maximal LR-subgroup of $A$.
 By Theorem~\ref{the:main} there exists a group $N$ with $G\le N \le A$
 and $|N:G| = 2$, and so by Lemma~\ref{lem:index2}, $(\Gamma,\C)$ is self-dual,  a contradiction. This shows that $(\Gamma,\C)$ is not a counterexample.

We now turn to Corollary~\ref{cor:cor}. Suppose that $(\Gamma,\mathcal C)$ is an LR-structure such that $ \{ \G , \R\}$ is the (unique) partition of $\C$ into cycles satisfying Definition~\ref{def:LR}, and that  $G=\Aut^+(\Gamma,\mathcal C)$ is contained in an arc-transitive subgroup $Y$. If $G$ is normal in $Y$, then since $\{\G,\R\}$ is a partition of $E(\Gamma)$, there exists $g\in Y$ such that $\G^g = \R$ and $\R^g = \G$ and hence $(\Gamma,\C)$ is self-dual. If $G$ is not normal in $Y$, then as above, for $y\in Y \setminus N_Y(G)$ the group $\langle G, G^y \rangle$ is a $2$-arc-transitive discrete subgroup of $\Aut(\Gamma)$ and then Theorem~\ref{the:solution} shows that there is an arc-transitive subgroup $X$ of $\Aut(\Gamma)$ such that $G$ is a normal subgroup of $X$ of index $2$, and hence $(\Gamma,\C)$ is self-dual.

\end{document}